\documentclass[preprint,12pt]{elsarticle}

\usepackage{amsmath,amssymb,amsthm}
\usepackage{tikz}
\usepackage{booktabs}
\usetikzlibrary{arrows.meta, positioning}
\usepackage[font=footnotesize]{caption}
\usepackage{graphicx}
\usepackage{makecell}
\usepackage{hyperref}
\usepackage{titlesec}
\titleformat{\paragraph}[runin]{\normalfont\bfseries}{}{0em}{}[]

\newtheorem{theorem}{Theorem}
\newtheorem{corollary}{Corollary}
\newtheorem{lemma}{Lemma}
\theoremstyle{definition}
\newtheorem{definition}{Definition}

\journal{[Journal Name]}

\begin{document}

\begin{frontmatter}

\title{Absorbing States of Binary Trust Gossip Are\\ Counted by Plane Partitions}

\author{Nicholas Boichuk}
\address{Independent Researcher}
\ead{nicholas.boichuk3@nhs.net}

\begin{abstract}
We study an opinion dynamics model in which $n$ agents hold directed
trust or distrust opinions about one another, represented as a matrix
$M \in \{0,1\}^{n \times n}$ in which 1 represents trust and 0
represents distrust. A gossip event $(a, z, y)$ causes agent $z$ to
adopt agent $a$'s opinion of $y$, provided that $z$ trusts $a$. We
characterize the absorbing states of this process, i.e.\ the states
in which no further opinion change can take place: we find that they
are the states in which agents are partitioned into isolated factions,
each faction containing a subset of core members who share mutual
trust, while the remaining peripheral members trust all core members
but receive no trust in return. This structure establishes a bijection
between absorbing states on $[n]$ and pairs consisting of a set
partition $\pi$ of $[n]$ together with a choice of non-empty subset
of each faction of $\pi$. The number of such absorbing states is
therefore given by OEIS A143405, with exponential generating function
$\exp(\exp(x) \cdot (\exp(x) - 1))$. In addition, up to isomorphism,
the count equals the number of plane partitions of $n$, given by OEIS
A000219, recovering MacMahon's classical product formula
$\prod_{k \geq 1} 1/(1 - x^k)^k$. Exhaustive computation for
$n \leq 7$ confirms both counts.
\end{abstract}


\end{frontmatter}


\section{Introduction and Motivation}

The dynamics of information flow between discrete agents have
been studied from computational, mathematical and network science
perspectives among others. Classical models such as DeGroot (1974)
\cite{degroot1974} and Friedkin--Johnsen (1990) \cite{friedkin1990}
describe how agents update a single opinion, such as a preference or
belief, by averaging over their neighbours in a fixed social network.
More recent work on epistemic gossip protocols (van Ditmarsch et al.,
2017 \cite{vanditmarsch2017}; Cooper et al., 2019 \cite{cooper2019})
studies how agents share private information (``secrets'') through
pairwise communication, with the goal of reaching common knowledge.
In these two traditions, the network that mediates communication is
exogenous: who trusts whom is fixed in advance and unaffected by the
content of what is communicated.

There has been relatively little exploration of endogenous information
propagation, in which agents exchange opinions about other agents on
the same network, and where those opinions in turn determine whose
future communications will be accepted or ignored. This mechanism, of
opinion recursively influencing its own propagation, has the potential
to produce non-trivial, self-reinforcing dynamics that are absent from
models with fixed trust structures. The closest existing model is the
Leviathan of Deffuant, Carletti, and Huet (2013): this uses the
architecture of an $N \times N$ matrix of continuous opinions serving
simultaneously as the trust network, and treats this architecture to a
simulation-based analysis \cite{deffuant2013}. By restricting opinions
to binary values, we will show that it is possible to obtain a finite
state space amenable to exact combinatorial analysis.

In the simplest version of such a model, given a set $S$ composed of
$n$ agents, each agent will hold one binary opinion about every other
agent in the group: either trust or distrust, represented by either 1
or 0. In the case of $n = 3$, where we have agents $a$, $b$, and $c$,
to describe the system we require six variables:
\[
\begin{array}{ll}
a(b),\ a(c) \\
b(a),\ b(c) \\
c(a),\ c(b),
\end{array}
\]
where $a(b)$ is agent $a$'s opinion of agent $b$.

Given binary opinion and three agents, this yields $2^6 = 64$ possible
system configurations. These variables can be used to populate a
square opinion matrix $M_{n \times n}$ which defines the current state
of the system:
\[
\begin{bmatrix}
- & a(b) & a(c) \\
b(a) & - & b(c) \\
c(a) & c(b) & -
\end{bmatrix}
\]
Note that such a setup excludes the possibility of self-opinion, so
the diagonal corresponding to self-opinion is undefined.

The presence of $n \geq 3$ agents allows us to define the gossip
function, where agent $a$ (the speaker) gossips to agent $z$ (the
listener) about agent $y$ (the target), sharing the speaker's trust or
distrust of the target, which causes the listener to adopt the opinion
of the speaker about the target, but only if the listener trusts the
speaker---if the listener does not, the speaker's opinion is ignored.
More formally, for agents $a, z, y \in S$, the gossip function
$\mathrm{gos}(M, a, z, y)$ sets $M_{zy} \leftarrow M_{ay}$ if
$M_{za} = 1$, otherwise leaves $M_{zy}$ unchanged if $M_{za} = 0$.
The agents to which the gossip function gets applied are selected
randomly: at each step, any valid triple may be selected.

Note that when an opinion flips $z(y) = 0 \to 1$ then agent $z$
becomes open to the influence of agent $y$, and conversely when
$z(y) = 1 \to 0$ then agent $z$ becomes invulnerable to any further
influence from $y$---i.e.\ when $M_{za} = 0$, row $z$ cannot receive
any entries from row $a$, but when $M_{za} = 1$, row $z$ will become
populated by the entry in row $a$ along the column $y$, if $(a, z, y)$
were to be selected by the gossip function. This is the core endogenous
nature of the model, since there is no distinction between opinions
and trust of others' opinions. The trust matrix and opinion matrix
are one and the same.

Furthermore, we can say that a state is \emph{absorbing} if no further
gossip can change it, which is the case if and only if for all
$a, z, y \in S$, $M_{za} = 0$ or $M_{zy} = M_{ay}$, i.e.\ every agent either distrusts a given speaker, or already agrees with that speaker about every third party, with no exceptions across the whole set. This is a fixed
point of the system from which no further change is possible. Where
this condition is not met, i.e.\ when at least one agent trusts an
agent with whom it disagrees, that different opinion may yet be
adopted by the trusting agent, and therefore the opinion matrix can
still evolve.

The main result of this work is a complete characterization of these
absorbing states (Characterization Theorem), from which we derive the
exact counts of absorbing states on labeled agents (Corollaries~1
and~2) and the number of absorbing states up to relabelling of
agents (Corollary~3). Perhaps surprisingly, the number of absorbing
states up to isomorphism for $n$ agents coincides with OEIS sequence
A000219 (the plane partitions of $n$, first enumerated by MacMahon
\cite{macmahon1912,macmahon1916}) while the total number of absorbing
states for $n$ agents coincides with A143405, a different sequence
hitherto unconnected with plane partitions.


\section{Characterization of Absorbing States}

\begin{definition}
Given the opinion matrix $M$ of an absorbing state with set $S$ of
agents, define the \emph{mutual trust relation} $\sim$ on $S$ as:
$i \sim j$ if and only if $M_{ij} = 1$ and $M_{ji} = 1$. Note that
$\sim$ is symmetric by definition.
\end{definition}

\begin{lemma}\label{lem:equiv}
In any absorbing state, the mutual trust relation $\sim$ is an
equivalence relation.
\end{lemma}

\begin{proof}
Reflexivity is vacuous (diagonal entries are undefined, and no gossip
event involves self-opinions). Symmetry holds by definition.

To prove transitivity, assume $i \sim j$ and $i \sim k$; we must show
that this implies $j \sim k$, meaning $M_{jk} = 1$ and $M_{kj} = 1$.

Since $M$ describes an absorbing state, when agent $i$ tells $j$ what
$i$ thinks of $k$ (the gossip event $\mathrm{gos}(M, i, j, k)$), $M$
must remain unchanged. Agent $i$ trusts $k$ ($M_{ik} = 1$ from
$i \sim k$), and since agent $j$ trusts $i$ ($M_{ji} = 1$), the
absorbing condition requires that agent $j$ must already agree with
$i$ about $k$ ($M_{jk} = M_{ik}$), so $M_{jk} = 1$. The symmetric
argument via $\mathrm{gos}(M, i, k, j)$ gives $M_{kj} = 1$, which
together give $j \sim k$. In other words, agent $j$ ends up trusting $k$ because they trust $i$, who in turn trusts $k$, and conversely agent $k$ ends up trusting $j$ because they trust $i$, who in turn trusts $j$.
\end{proof}

\begin{definition}
Given an absorbing state with set $S$ of agents, define the
\emph{one-way trust relation} $\rhd$ on $S$ as: for two distinct
agents $i, j \in S$, $i \rhd j$ if and only if $M_{ij} = 1$ and
$M_{ji} = 0$. Agent $i$ is said to have outgoing trust to $j$, while
agent $j$ is said to have incoming trust from $i$.
\end{definition}

\begin{lemma}\label{lem:exclusive}
In any absorbing state, an agent can have
either mutual trust relations or outgoing one-way trust relations,
but not both.
\end{lemma}

\begin{proof}
Suppose $i \sim j$ and $i \rhd k$ in an absorbing state.
$M_{ji} = 1$ and $M_{ik} = 1$, therefore the gossip event
$\mathrm{gos}(M, i, j, k)$ (agent $i$ tells $j$ what $i$ thinks of
$k$) requires that $j$ also trusts $k$ ($M_{jk} = 1$). The further
gossip event $\mathrm{gos}(M, k, j, i)$ (agent $k$ tells $j$ what
$k$ thinks of $i$) requires $j$'s opinion of $i$ to be 0
($M_{ji} = M_{ki} = 0$), contradicting $M_{ji} = 1$ and violating
the $i \sim j$ assumption. In other words, agent $i$ undermines
itself by trusting $k$ (an agent who does not trust them back),
passing that trust on to an agent $j$ who does initially trust them
back, but becomes open to $k$'s influence and stops trusting $i$,
leaving $i$ distrusted by both, and demonstrating that the opinion
system had in fact not reached stability.
\end{proof}

\begin{lemma}\label{lem:cascade}
In any absorbing state, every agent $i$ that has at least one outgoing
trust relation to $j$ will also have an outgoing trust relation to
every agent that $j$ has a mutual trust relation with, i.e.\ if
$i \rhd j$ and $j \sim k$, then $i \rhd k$.
\end{lemma}

\begin{proof}
Since $i$ trusts $j$ ($M_{ij} = 1$) and $j$ trusts $k$
($M_{jk} = 1$), the gossip event $\mathrm{gos}(M, j, i, k)$ (agent
$j$ tells $i$ what $j$ thinks of $k$) requires that agent $i$'s
opinion of $k$ be identical to $j$'s ($M_{ik} = M_{jk}$) under the
absorbing condition, so $M_{ik} = 1$. From Lemma~\ref{lem:exclusive}, $i \sim k$
cannot be true since $i \rhd j$ is already true, so $M_{ki} = 0$,
therefore $i \rhd k$. In other words, agent $i$ trusts $j$ and therefore adopts their trust of $k$, while agent $k$ also trusts $j$ and therefore adopts their distrust of $i$.
\end{proof}

\begin{lemma}\label{lem:targets}
In any absorbing state with set $S$ of agents, if agent $i$ has outgoing trust relations to
two distinct agents $j, k \in S$, then the two agents must share
mutual trust, i.e.\ if $i \rhd j$ and $i \rhd k$, then $j \sim k$.
\end{lemma}

\begin{proof}
If agent $i$ trusts both $j$ and $k$ ($M_{ij} = 1$ and $M_{ik} = 1$),
then under the absorbing condition the gossip event
$\mathrm{gos}(M, j, i, k)$ (agent $j$ tells $i$ what $j$ thinks of
$k$) requires that $M_{jk} = M_{ik} = 1$, and equivalently the gossip
event $\mathrm{gos}(M, k, i, j)$ (agent $k$ tells $i$ what $k$
thinks of $j$) requires that $M_{kj} = M_{ij} = 1$, therefore
$j \sim k$. In other words, agents $j$ and $k$ must have mutual trust if agent $i$ trusts them both, since otherwise $j$'s distrust of $k$ would be adopted by $i$ or $k$'s distrust of $j$ would be adopted by $i$.
\end{proof}

\begin{lemma}\label{lem:core-targets}
In any absorbing state, an agent who has at least one outgoing one-way trust relation cannot have any incoming one-way trust relations, i.e. if $j \rhd k$ then there is no $i$ such that $i \rhd j$.
\end{lemma}

\begin{proof}
Suppose agent $i$ trusts agent $j$ who has at least one outgoing one-way trust relation $j \rhd k$
for some agent $k$. Since $i$ trusts $j$, the gossip event
$\mathrm{gos}(M, j, i, k)$ (agent $j$ tells $i$ what $j$ thinks of $k$) requires $M_{ik} = M_{jk} = 1$, so $i$ trusts both $j$ and $k$.
By Lemma~\ref{lem:targets}, $j \sim k$. But $j \rhd k$ requires
$M_{kj} = 0$, contradicting $j \sim k$. In other words, agent $i$ trusts $j$ and therefore adopts $j$'s trust of $k$ and adopts $k$'s opinions, but since $k$ does not trust $j$ back, $i$ must lose trust in $j$, demonstrating that the opinion system had in fact not reached stability.
\end{proof}

\begin{definition}
Agents that have at least one outgoing one-way trust relation are called
\emph{peripheral members}, who receive no trust from any agent, whether mutual (by Lemma~\ref{lem:exclusive}) or incoming one-way (by Lemma~\ref{lem:core-targets}). Conversely, agents are called \emph{core members} if they have no outgoing one-way trust relations---they may have \emph{incoming} one-way trust relations, mutual trust relations, or no trust relations whatsoever.
\end{definition}

\begin{theorem}[Characterization]
An opinion matrix $M \in \{0,1\}^{n \times n}$ of set $S$ of $n$
agents describes an absorbing state if and only if there exists a
partition $\pi$ of $S$ into non-empty factions $F_1, \ldots, F_p$
where $p \leq n$ and a non-empty subset of core members
$C_\ell \subseteq F_\ell$ for each $\ell$, such that:
\begin{quote}
\textbf{Condition (i):} $M_{ij} = 1$ if and only if $i, j \in F_\ell$
and $j \in C_\ell$ (i.e.\ $i$ and $j$ belong to the same faction and
$j$ is a core member thereof).
\end{quote}
In other words, the set $S$ of $n$ agents contains up to $n$ factions,
each faction containing at least one core member, and no agent trusts
any agent except the core members of their own faction.
\end{theorem}

\begin{proof}[Proof of sufficiency]
Let $M$ satisfy condition~(i). To verify the absorbing condition for
every $\mathrm{gos}(M, a, z, y)$:

If agent $a$ is not a core member of $z$'s faction (whether $z$ is core
or peripheral), either because $a$
is peripheral in that faction ($a, z \in F_\ell$,
$a \notin C_\ell$), or because $a$ belongs to a different faction
($z \in F_\ell$, $a \notin F_\ell$), then $M_{za} = 0$ by
condition~(i), so agent $z$ cannot receive any opinions from $a$ and
the gossip event leaves $M$ unchanged.

If agent $a$ is a core member of $z$'s faction (whether $z$ is core
or peripheral) ($a \in C_\ell$, $z \in F_\ell$), then $M_{za} = 1$
and we must verify $M_{zy} = M_{ay}$. By condition~(i), both agents
$a$ and $z$ assign opinion 1 to core members of their shared faction
and opinion 0 to all other agents. Since agents $a$ and $z$ belong to
the same faction, their rows agree on every column, so
$M_{zy} = M_{ay}$.
\end{proof}

\begin{proof}[Proof of necessity]
Let $M$ describe an absorbing state. By Lemma~\ref{lem:equiv}, the mutual trust
relation $\sim$ is an equivalence relation on $S$. By Definition~3,
every agent is either a core member (participating in some equivalence
class of $\sim$ with at least one other agent, or having no mutual or
outgoing trust relations at all) or a peripheral member (having at
least one outgoing one-way trust relation). By Lemmas~\ref{lem:cascade}, \ref{lem:targets}, and~\ref{lem:core-targets}, a peripheral agent's trust targets belong to a single core: they must all be mutually trusting (Lemma~\ref{lem:targets}), the agent must trust all of them (Lemma~\ref{lem:cascade}), and none of them can be peripheral (Lemma~\ref{lem:core-targets}). Assign each peripheral agent to the faction of the core it trusts; assign agents with no trust relations at all to singleton factions
in which they are the sole core.

Now to verify Condition~(i) under this arrangement. If $i$ and $j$
are in the same faction and $j$ is a core member, then either $i$ is
also a core member (so $i \sim j$, giving $M_{ij} = 1$) or $i$ is
peripheral and trusts all core members of its faction (giving
$M_{ij} = 1$). Conversely, if $M_{ij} = 1$, then either $i \sim j$
(same core, same faction) or $i \rhd j$ (peripheral $i$ trusts core
$j$, same faction by construction). If $i$ and $j$ are in different
factions, then $j$ is not a trust target of $i$ by construction, so $M_{ij} = 0$.
\end{proof}

The preceding proof implies four possible faction structures, with size up to $n$ and with a minimum of one core agent but otherwise having no constraints on the numbers of core and peripheral agents:

\begin{figure}[t]
  \centering
  \begin{tikzpicture}[
      >=Stealth,
      core/.style={circle, fill={rgb,255:red,83;green,74;blue,183}, text=white, 
                   minimum size=16pt, inner sep=0pt, font=\scriptsize},
      periph/.style={circle, draw=gray, line width=0.7pt, 
                     minimum size=16pt, inner sep=0pt, font=\scriptsize},
      faction/.style={rounded corners=12pt, draw=gray, dashed, line width=0.7pt},
      mutual/.style={->, draw={rgb,255:red,83;green,74;blue,183}, line width=0.6pt,
                     >={Stealth[length=4pt]}},
      oneway/.style={->, draw=gray, line width=0.6pt, >={Stealth[length=4pt]}},
      flabel/.style={font=\small\bfseries},
  ]

  \draw[faction] (7.2, -3.0) rectangle (11.4, 0.4);
  \node[flabel] at (9.3, 0.05) {Faction 4};

  \node[core] (c1) at (8.5, -0.8) {$c_6$};
  \node[core] (c2) at (10.1, -0.8) {$c_7$};

  \draw[mutual] ([yshift=2.5pt]c1.east) -- ([yshift=2.5pt]c2.west);
  \draw[mutual] ([yshift=-2.5pt]c2.west) -- ([yshift=-2.5pt]c1.east);

  \node[periph] (p1) at (8.5, -2.2) {$p_2$};
  \node[periph] (p2) at (10.1, -2.2) {$p_3$};

  \draw[oneway] (p1.north) -- (c1.south);
  \draw[oneway] ([xshift=3pt]p1.north) -- ([xshift=-3pt]c2.south);
  \draw[oneway] ([xshift=-3pt]p2.north) -- ([xshift=3pt]c1.south);
  \draw[oneway] (p2.north) -- (c2.south);

  \draw[faction] (2.5, -3.0) rectangle (6.3, 0.4);
  \node[flabel] at (4.4, 0.05) {Faction 2};

  \node[core] (c3) at (4.4, -0.7) {$c_2$};
  \node[core] (c4) at (3.4, -2.43) {$c_3$};
  \node[core] (c5) at (5.5, -2.43) {$c_4$};

  \draw[mutual] ([shift={(-0.075cm,0.043cm)}]c3.south west) -- 
                ([shift={(-0.075cm,0.043cm)}]c4.north east);
  \draw[mutual] ([shift={(0.075cm,-0.043cm)}]c4.north east) -- 
                ([shift={(0.075cm,-0.043cm)}]c3.south west);

  \draw[mutual] ([shift={(0.075cm,0.043cm)}]c3.south east) -- 
                ([shift={(0.075cm,0.043cm)}]c5.north west);
  \draw[mutual] ([shift={(-0.075cm,-0.043cm)}]c5.north west) -- 
                ([shift={(-0.075cm,-0.043cm)}]c3.south east);

  \draw[mutual] ([yshift=2.5pt]c4.east) -- ([yshift=2.5pt]c5.west);
  \draw[mutual] ([yshift=-2.5pt]c5.west) -- ([yshift=-2.5pt]c4.east);

  \draw[faction] (-0.6, -1.2) rectangle (1.6, 0.4);
  \node[flabel] at (0.5, -0.05) {Faction 1};

  \node[core] (c6) at (0.5, -0.7) {$c_1$};

  \draw[faction] (-0.6, -3.0) rectangle (1.6, -1.5);
  \node[flabel] at (0.5, -1.95) {Faction 3};

  \node[core] (c7) at (0.05, -2.55) {$c_5$};
  \node[periph] (p3) at (0.95, -2.55) {$p_1$};

  \draw[oneway] (p3.west) -- (c7.east);

  \node[core, minimum size=10pt] at (0.3, -3.8) {};
  \node[anchor=west, font=\footnotesize] at (0.6, -3.8) {Core member};

  \node[periph, minimum size=10pt] at (0.3, -4.3) {};
  \node[anchor=west, font=\footnotesize] at (0.6, -4.3) {Peripheral member};

  \draw[mutual] (4.6, -3.8) -- (5.3, -3.8);
  \node[anchor=west, font=\footnotesize] at (5.45, -3.8) {Mutual trust};

  \draw[oneway] (4.6, -4.3) -- (5.3, -4.3);
  \node[anchor=west, font=\footnotesize] at (5.45, -4.3) {One-way trust};

  \end{tikzpicture}
  \captionsetup{font=footnotesize}
  \caption{An example absorbing state with partition
  $\pi = \{\{c_1\},\; \{c_2, c_3, c_4\},\; \{c_5, p_1\},\; \{c_6, c_7, p_2, p_3\}\}$
  and selected cores $\{c_1\}$, $\{c_2, c_3, c_4\}$, $\{c_5\}$, $\{c_6, c_7\}$.
  Violet nodes are core members; white nodes are peripheral.}
  \label{fig:factions}
\end{figure}
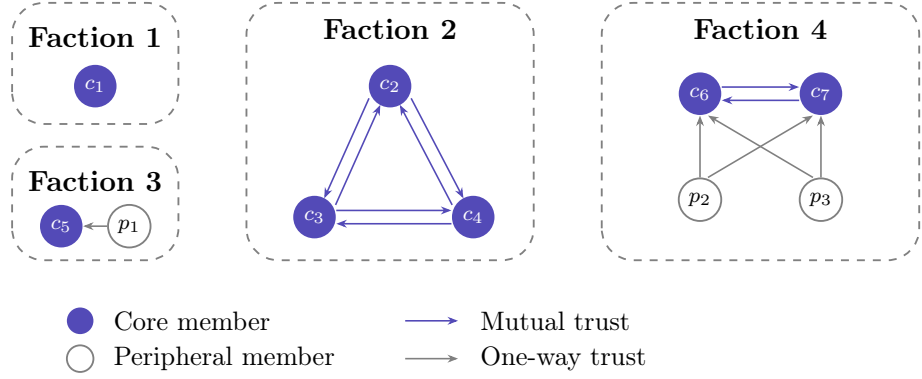

\noindent\textbf{Case~1:} a faction $F$ composed of a single agent
$i \in C_\ell$ who trusts no other agent ($C_\ell = F_\ell$)---the singleton faction. This
agent by definition has no outgoing trust relations, defining them as
a core member and the sole agent in their faction.

\noindent\textbf{Case~2:} a faction $F$ composed only of $m$ core members
$i_1, i_2, \ldots, i_m \in C_\ell$ who all share mutual trust
relations with each other ($C_\ell = F_\ell$).

\noindent\textbf{Case~3:} a faction $F$ composed of $k$ total members,
including exactly one core member $i \in C_\ell$ and $k - 1$
peripheral members $j_1, j_2, \ldots, j_{k-1} \notin C_\ell$. The
one core member trusts no other agent, and the peripheral members
only trust the one core member in their faction
($j_1, j_2, \ldots, j_{k-1} \rhd i$).

\noindent\textbf{Case~4:} a faction $F$ composed of $k$ total members,
including $m > 1$ core members
$i_1, i_2, \ldots, i_m \in C_\ell$ and $k - m$ peripheral members
$j_1, j_2, \ldots, j_{k-m} \notin C_\ell$. The core members trust
only each other, while the peripheral members
only trust the core members ($j_1, j_2, \ldots, j_{k-m} \rhd i_1, i_2, \ldots, i_m$).

A faction $F$ can be composed purely of core members, but not purely of peripheral members: any of its peripheral members who trust no other agents would themselves be core members of their own singleton factions, while any peripheral members who trust some other agent would turn that agent into a core member and fall into its faction.


\section{Enumeration}

The Characterization Theorem reduces the problem of counting absorbing
states to a combinatorial one: each absorbing state on set $S$ of $n$
agents corresponds to a choice of set partition (the factions) together with a choice
of non-empty subset within each faction (the core). We now derive exact counts for
both the labeled and unlabeled cases.

\begin{corollary}[Labeled count]
The number of absorbing states on $n$ labeled agents is
\begin{equation}
a(n) = \sum_{\pi} \prod_{F \in \pi} (2^{|F|} - 1)
\end{equation}
where the sum ranges over all set partitions $\pi$ of $S$, and the
product runs over the factions $F$ of $\pi$.
\end{corollary}

\begin{proof}
A faction of size $k$ admits $2^k - 1$ non-empty subsets. Since the core choices are
independent across factions, the number of absorbing states arising
from a given partition $\pi$ is the product of $(2^{|F|} - 1)$ over
all factions $F$. Summing over all set partitions gives the total
count.

As a small example, for $n = 3$ with agents $\{1, 2, 3\}$, the five
set partitions contribute:
\begin{align*}
\{\{1, 2, 3\}\}: &\quad (2^3 - 1) = 7 \\
\{\{1, 2\}, \{3\}\}: &\quad (2^2 - 1)(2^1 - 1) = 3 \\
\{\{1, 3\}, \{2\}\}: &\quad (2^2 - 1)(2^1 - 1) = 3 \\
\{\{2, 3\}, \{1\}\}: &\quad (2^2 - 1)(2^1 - 1) = 3 \\
\{\{1\}, \{2\}, \{3\}\}: &\quad (2^1 - 1)^3 = 1 \\
& \hspace{2.8cm} \text{Total: } 17
\end{align*}
which matches the exhaustive computation in Section~4.
\end{proof}

\begin{corollary}[Exponential generating function]
The exponential generating function for the sequence $(a(n))_{n \geq 0}$ is
\begin{equation}
\sum_{n \geq 0} a(n) \frac{x^n}{n!}
= \exp\!\left(\sum_{k \geq 1} (2^k - 1) \frac{x^k}{k!}\right)
= \exp(e^x(e^x - 1)).
\end{equation}
\end{corollary}

\begin{proof}
This is an application of the exponential formula from Stanley
\cite{stanley1999}. When a structure on $S$ is built by partitioning
$S$ into non-empty factions and independently decorating each faction
of size $k$ in $F_k$ ways, the EGF for the total count is
\begin{equation}
\exp\!\left(\sum_{k \geq 1} F_k \frac{x^k}{k!}\right).
\end{equation}

Here $F_k = 2^k - 1$. To simplify the inner sum:
\begin{equation}
\sum_{k \geq 1} (2^k - 1) \frac{x^k}{k!}
= \sum_{k \geq 1} \frac{(2x)^k}{k!}
  - \sum_{k \geq 1} \frac{x^k}{k!}
= (e^{2x} - 1) - (e^x - 1)
= e^{2x} - e^x
= e^x(e^x - 1).
\end{equation}
Exponentiating gives $\exp(e^x(e^x - 1))$. This is the EGF for OEIS
A143405, whose asymptotics were studied by Kotesovec
\cite{kotesovec2022}.
\end{proof}

\begin{corollary}[Unlabeled count]
The number of absorbing states on $n$ agents up to relabeling is equal
to the number of plane partitions of $n$ (OEIS A000219).
\end{corollary}

\begin{proof}
Two absorbing states are isomorphic if and only if one can be obtained
from the other by a permutation of the agent labels. By the
Characterization Theorem, an absorbing state has factional structure,
which is fully determined by sizes of the factions and the number of
core members within each. Permuting agent labels can turn any faction
of size $k$ with $m$ core members into any other faction of the same
size with the same number of core members. Two absorbing states are
therefore isomorphic if and only if they have the same multiset of
pairs $(k, m)$, where $k$ is the faction size and
$m \in \{1, \ldots, k\}$ is the number of core members, for each
faction in that state.

The isomorphism classes are therefore in bijection with integer partitions of $n$ in which each part of size $k$ comes in $k$ distinguishable types, indexed by the core size $m$. The number of types of faction
equals the faction size: a faction of size $k$ admits exactly $k$
distinct core configurations (one for each core size
$m = 1, 2, \ldots, k$).

To count the number of such partitions: for each size $k$, the $k$
types of factions each contribute with generating function
$1/(1 - x^k)$ (encoding the choice of ``how many factions of this
type?''), and the $k$ types produce the $k$-fold product
$1/(1 - x^k)^k$. Since the $k$ types of faction contribute
independently their counts are multiplied, so the total count of
partitions is given by
\begin{equation}
\prod_{k \geq 1} \frac{1}{(1 - x^k)^k}.
\end{equation}
This is the generating function for plane partitions of $n$ from
MacMahon \cite{macmahon1912,macmahon1916} and Andrews \cite{andrews1976}.
\end{proof}

We note that this also establishes a previously unrecorded connection
between sequences A143405 and A000219 in the OEIS: the former is the
labeled count and the latter the unlabeled count of the same
combinatorial structure: set partitions with a distinguished non-empty
subset per block.


\section{Computational Verification}

We independently verified the Characterization Theorem and its
corollaries by exhaustive enumeration of the state space for $n \leq 7$.

For each value of $n$, every binary matrix $M$ in
$\{0,1\}^{n \times n}$ (with undefined diagonal) was tested against
the absorbing condition: for all triples $(a, z, y)$ of distinct
agents, either $M_{za} = 0$ or $M_{zy} = M_{ay}$. This requires
checking $O(n^3)$ triples per state across a state space of
$2^{n(n-1)}$ matrices.

For $n \leq 6$, the enumeration was performed in C on a single CPU
thread; the $n = 6$ case ($2^{30} \approx 1.07 \times 10^9$ states)
completed in under 4 seconds. For $n = 7$, the state space grows to
$2^{42} \approx 4.4 \times 10^{12}$ states, requiring GPU
acceleration. This computation was performed on an NVIDIA T4 GPU using
a CUDA implementation, taking 132 seconds.

For the unlabeled counts, absorbing states were classified up to
isomorphism under the action of the symmetric group $S_n$. The
canonical form of each state was computed by applying all $n!$
permutations to the adjacency matrix and taking the lexicographic
minimum. This was performed for $n \leq 7$ using a combined CUDA~+~CPU
program: the GPU identified the absorbing states, and the CPU
classified them, in a total of 253 seconds.

Additionally, for $n = 3$ and $n = 4$, the complete state transition
graph was constructed: every state was connected to all states
reachable by a single gossip event (Figure~2 shows the $n = 3$ case). Strongly connected
components (SCCs) were identified via Tarjan's algorithm, and
absorbing states were extracted as SCCs with no outgoing edges. Every state was found to eventually reach a fixed point; no limit cycles were detected.

\begin{figure}[t]
  \centering
  \scalebox{0.8}{%
  \begin{tikzpicture}[
      >=Stealth,
      absbox/.style={
          rounded corners=3pt,
          fill={rgb,255:red,83;green,74;blue,183},
          text=white, font=\footnotesize,
          minimum width=1.3cm, minimum height=1.1cm,
          inner sep=3pt, align=center,
      },
      transbox/.style={
          rounded corners=3pt,
          draw=gray, dashed, line width=0.8pt,
          fill={rgb,255:red,241;green,239;blue,232},
          text=gray!70!black, font=\footnotesize,
          minimum width=1.3cm, minimum height=1.1cm,
          inner sep=3pt, align=center,
      },
      elabel/.style={font=\footnotesize, anchor=north},
      rlabel/.style={font=\footnotesize\bfseries, text=gray, anchor=east},
      tarr/.style={
          draw=gray!50, line width=1pt,
          {Stealth[length=4pt]}-{Stealth[length=4pt]},
      },
      toarr/.style={
          draw=gray!50, line width=1pt,
          -{Stealth[length=4pt]},
      },
      abarr/.style={
          draw={rgb,255:red,83;green,74;blue,183},
          line width=1pt, dashed,
          -{Stealth[length=4pt]},
      },
  ]
 
  \def\xA{1.2}
  \def\xB{3.2}
  \def\xC{5.2}
  \def\xD{7.2}
  \def\xE{9.2}
  \def\xF{11.2}
  \def\xG{13.2}
  \def\ytrans{0}
  \def\yabs{-4.0}
 
  \node[rlabel] at (-0.2, \ytrans) {Transient};
  \node[rlabel] at (-0.2, \yabs) {Absorbing};
 
  \node[elabel] at (\xA, {\yabs-0.85}) {0};
  \node[elabel] at (\xB, {\yabs-0.85}) {1};
  \node[elabel] at (\xC, {\yabs-0.85}) {2};
  \node[elabel] at (\xD, {\yabs-0.85}) {3};
  \node[elabel] at (\xE, {\yabs-0.85}) {4};
  \node[elabel] at (\xF, {\yabs-0.85}) {5};
  \node[elabel] at (\xG, {\yabs-0.85}) {6};
  \node[font=\footnotesize, text=gray, anchor=north]
      at (7.2, {\yabs-1.2}) {Edge count};
 
  \node[absbox] (A0) at (\xA, \yabs) {1\\fixed\\point};
  \node[absbox] (A1) at (\xB, \yabs) {6\\fixed\\points};
  \node[absbox] (A2) at (\xC, \yabs) {6\\fixed\\points};
  \node[absbox] (A4) at (\xE, \yabs) {3\\fixed\\points};
  \node[absbox] (A6) at (\xG, \yabs) {1\\fixed\\point};
 
  \node[transbox] (T2) at (\xC, \ytrans) {9\\transient};
  \node[transbox] (T3) at (\xD, \ytrans) {20\\transient};
  \node[transbox] (T4) at (\xE, \ytrans) {12\\transient};
  \node[transbox] (T5) at (\xF, \ytrans) {6\\transient};
 
  \draw[tarr]  (T2.east) -- (T3.west);
  \draw[tarr]  (T3.east) -- (T4.west);
  \draw[toarr] (T4.east) -- (T5.west);
 
  \draw[abarr] (T2.south) -- (A1.north);
  \draw[abarr] ([xshift=-4pt]T3.south) -- ([xshift=4pt]A2.north);
  \draw[abarr] ([xshift=4pt]T3.south) -- ([xshift=-4pt]A4.north);
  \draw[abarr] ([xshift=-4pt]T5.south) -- ([xshift=4pt]A4.north);
  \draw[abarr] ([xshift=4pt]T5.south) -- ([xshift=-4pt]A6.north);
 
  \draw[tarr] (0.3, 1.5) -- (1.3, 1.5);
  \node[font=\footnotesize, anchor=west] at (1.5, 1.5)
      {Between transient states};
  \draw[abarr] (0.3, 0.9) -- (1.3, 0.9);
  \node[font=\footnotesize, anchor=west] at (1.5, 0.9)
      {Absorption into fixed point};
  \node[font=\footnotesize\itshape, text=gray!70!black,
        anchor=west, text width=6cm] at (7.5, 1.2)
      {All 47 transient states reach one of the
       17 fixed points. No limit cycles.};
 
  \end{tikzpicture}}
  \caption{State space for $n=3$: all 64 states partitioned into
  47 transient states and 17 fixed points, grouped by trust-edge count.
  Arrows indicate possible transitions; no limit cycles exist.}
  \label{fig:statespace}
\end{figure}
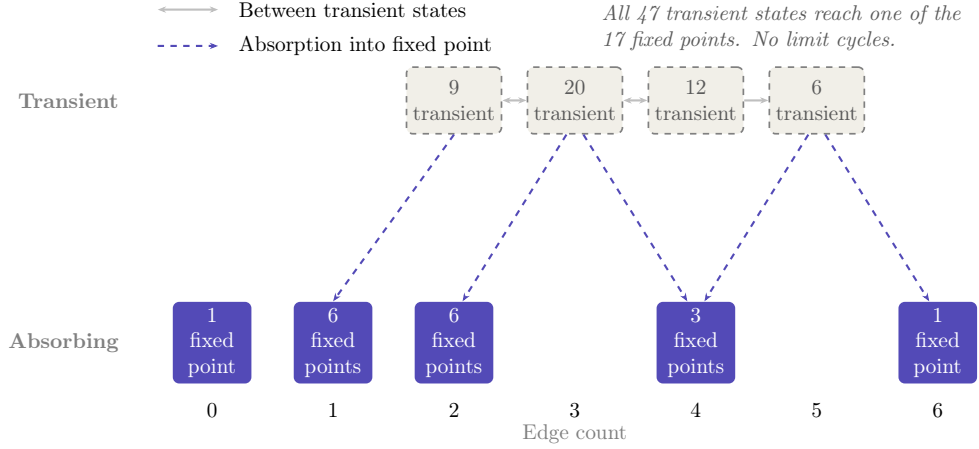

The results of computational verification are summarized in Table~1.

\begin{table}[ht]
\centering
\caption{Exhaustive verification of absorbing state counts for $n \leq 7$.}
\label{tab:verification}
\small
\begin{tabular}{@{}rrrrrr@{}}
\toprule
$n$ & $2^{n(n-1)}$ & \makecell{Absorbing\\(labeled)} & A143405$(n)$
    & \makecell{Absorbing\\(unlabeled)} & A000219$(n)$ \\
\midrule
1 & 1                          &      1 &      1 &  1 &  1 \\
2 & 4                          &      4 &      4 &  3 &  3 \\
3 & 64                         &     17 &     17 &  6 &  6 \\
4 & 4{,}096                    &     89 &     89 & 13 & 13 \\
5 & 1{,}048{,}576               &    552 &    552 & 24 & 24 \\
6 & 1{,}073{,}741{,}824         &  3{,}895 &  3{,}895 & 48 & 48 \\
7 & 4{,}398{,}046{,}511{,}104   & 30{,}641 & 30{,}641 & 86 & 86 \\
\bottomrule
\end{tabular}
\end{table}

In all cases where both values were computed, the labeled counts match
A143405 and the unlabeled counts match A000219 exactly.

Source code for the CUDA implementation of the combined classification program is available at \href{https://github.com/scharmcrab/absorbing-states-enumeration}{GitHub ("absorbing-state-enumerations", user scharmcrab)}.

More efficient enumeration, perhaps applying McKay's algorithm for
canonical graph labelling \cite{mckay2014}, could extend the
verification range beyond $n=7$.


\section{Discussion}

The identification of binary gossip absorbing states with plane partitions connects a simple social dynamics model to a classical object in enumerative combinatorics; this connection is, to our knowledge, new.

\paragraph{Relation of binary gossip model to other gossip models}
The epistemic gossip protocols studied by van Ditmarsch and Apt
\cite{apt2016,vanditmarsch2017} use a monotone update rule (secrets,
once learned, are never forgotten), yielding the unique absorbing
state of full knowledge. In contrast, the non-monotone dynamics studied here, where
trust can be both gained and lost, produce an entire
landscape of absorbing states. The Boolean gossip networks of Li et al.\ \cite{li2018} study binary-valued consensus with Boolean update functions on a fixed communication graph; the structure of their equilibria depends on the topology of the chosen graph, whereas in our model all agents can communicate and the equilibrium count depends only on $n$.

The gossip function $\mathrm{gos}(M, a, z, y)$ studied here encodes
the triadic logic identified by Heider \cite{heider1946} and
formalized in graph-theoretic terms by Cartwright and Harary
\cite{cartwright1956}: trust in a speaker causes the listener to
inherit the speaker's opinion of a third party. Marvel et al.\
\cite{marvel2011} studied a continuous-time version of this principle
and proved that generic initial conditions converge to at most two
hostile factions. Our discrete binary model admits richer equilibria:
multiple factions with internal core-periphery structure, a difference
attributable to the directed, asymmetric nature of trust in our
formulation.

Kawakatsu, Kessinger, and Plotkin \cite{kawakatsu2024} study a closely
related model in the evolutionary biology literature: agents hold
binary (Good/Bad) opinions about each other in an $n \times n$ matrix,
and gossip copies one agent's opinion to another. The critical
difference is that in their model, the listener unconditionally adopts
the speaker's view regardless of trust. The trust-gating condition
($M_{za} = 1$) is what produces the faction structure in our model,
and without it the dynamics do not generate isolated factions with
core-periphery hierarchies.

Jia, Friedkin, and Bullo \cite{jia2015, jia2016} study the coevolution of
opinions and interpersonal appraisals, showing convergence to
what they describe as ``factions with followers,'' which is
the closest qualitative match to our core-peripheral factions. The
mathematical setting differs substantially, as they use
continuous-valued appraisals and separate opinions and influence
weights. In addition, they characterize equilibrium states but
do not enumerate them; the combinatorial counting results of Section~3
appear to have no counterpart in the opinion dynamics literature.

The formation of factions and hierarchies
is a recurring theme across these diverse models of
opinion flow, whether continuous or discrete, conditional or
unconditional, analytical or simulation-based. The contribution of the
present work is not the qualitative observation of faction or hierarchy, but the
exact combinatorial characterization thereof.

\paragraph{Open problems}
Several directions remain:

\emph{Convergence.}
We have characterized the structure and number of the absorbing states, but not which
absorbing states are more likely to be reached or how quickly, or
whether any initial state can lead to any absorbing state. We can
compute which initial conditions lead to which equilibria for small
$n$ (Section~4) but we lack a general theory.

\emph{Limit cycles.}
Exhaustive computation for $n \leq 4$ reveals no limit cycles: every
state eventually reaches a fixed point. We conjecture this holds for
all $n$, but a proof is not in hand.

\emph{Extensions of the model.}
The gossip model we have been exploring here can be said to be a
binary, 1st-order, random gossip model:
\begin{enumerate}
\item[(A)] \emph{Binary:} opinion values are strictly either 0 or 1,
  i.e.\ $M_{ij} \in \{0, 1\}$.
\item[(B)] \emph{1st-order:} only 1st-order opinion values of form
  $a(b)$ are stored and exchanged.
\item[(C)] \emph{Random gossip:} gossip triples are selected
  at random.
\end{enumerate}
This model is a special case which can be extended by allowing opinion
to take on values in continuous range $M_{ij} \in [0, 1]$, by
augmenting the agents' theory of mind, allowing them to not only
communicate their own opinion, e.g.\ $a(b)$, but also hearsay from
others, e.g.\ $a(b(c))$, and by selecting gossip triples non-randomly
(for example biasing agents towards gossiping with agents they trust).
The absorbing states of continuous, higher-order and non-random
communication models are not yet characterized or mapped to a known
combinatorial object; our initial exploration may serve as the first step to a full description of a broader class of model.

\bigskip
The Characterization Theorem reveals that the equilibrium of binary, 1st order, random gossip forces a specific social architecture: the population fragments into
isolated factions with no cross-faction trust. Each faction has a
strict core--periphery hierarchy, where peripheral agents
trust the core but the core does not reciprocate. This structure is
not imposed but rather emerges from simple initial rules as the only configuration stable
against any further gossip---factional hierarchy is an organic consequence
of this form of trust-based information transmission. But unlike in prior models, the number of such configurations can be determined exactly, and equals the number of plane partitions of $n$.


\section*{Acknowledgements}

Formatted and illustrated using Overleaf.

No competing interests to declare.

This research did not receive any specific grant from funding agencies
in the public, commercial, or not-for-profit sectors.

\medskip
Use of AI tools: Anthropic Claude Opus 4.6 and Sonnet 4.6 were used
for assistance with drafting sections, conducting literature searches,
and coding tasks. All mathematical content was verified independently
by the author.



\begin{thebibliography}{24}

\bibitem{andrews1976}
G.E.~Andrews, \emph{The Theory of Partitions},
Addison-Wesley, Encyclopedia of Mathematics and its Applications,
Vol.~2, 1976. Cambridge University Press paperback reprint, 1998.

\bibitem{apt2016}
K.R.~Apt, D.~Grossi, W.~van der Hoek,
Epistemic protocols for distributed gossiping,
in: Proceedings of TARK 2015, EPTCS 215, pp.~51--66, 2016.


\bibitem{cartwright1956}
D.~Cartwright, F.~Harary,
Structural balance: A generalization of Heider's theory,
\emph{Psychological Review} 63\,(5) (1956) 277--293.

\bibitem{cooper2019}
M.C.~Cooper, A.~Herzig, F.~Maffre, F.~Maris, P.~R\'{e}gnier,
The epistemic gossip problem,
\emph{Discrete Mathematics} 342\,(3) (2019) 654--663.

\bibitem{degroot1974}
M.H.~DeGroot,
Reaching a consensus,
\emph{Journal of the American Statistical Association} 69\,(345)
(1974) 118--121.

\bibitem{deffuant2013}
G.~Deffuant, T.~Carletti, S.~Huet,
The Leviathan model: Absolute dominance, generalised distrust, small
worlds and other patterns emerging from combining vanity with opinion
propagation,
\emph{Journal of Artificial Societies and Social Simulation} 16\,(1)
(2013) 5.

\bibitem{friedkin1990}
N.E.~Friedkin, E.C.~Johnsen,
Social influence and opinions,
\emph{The Journal of Mathematical Sociology} 15\,(3--4) (1990)
193--206.

\bibitem{heider1946}
F.~Heider,
Attitudes and cognitive organization,
\emph{The Journal of Psychology} 21\,(1) (1946) 107--112.

\bibitem{jia2015}
P.~Jia, A.~MirTabatabaei, N.E.~Friedkin, F.~Bullo,
Opinion dynamics and the evolution of social power in influence
networks,
\emph{SIAM Review} 57\,(3) (2015) 367--397.

\bibitem{jia2016}
P.~Jia, N.E.~Friedkin, F.~Bullo,
The coevolution of appraisal and influence networks leads to
structural balance,
\emph{IEEE Trans.\ Network Science and Engineering} 3\,(4) (2016)
286--298.

\bibitem{kawakatsu2024}
M.~Kawakatsu, T.A.~Kessinger, J.B.~Plotkin,
A mechanistic model of gossip, reputations, and cooperation,
\emph{PNAS} 121\,(20) (2024) e2400689121.

\bibitem{kotesovec2022}
V.~Kotesovec,
Asymptotics for a certain group of exponential generating functions,
arXiv:2207.10568, 2022.

\bibitem{li2018}
B.~Li, J.~Wu, H.~Qi, A.~Proutiere, G.~Shi,
Boolean gossip networks,
\emph{IEEE/ACM Trans.\ Networking} 26\,(1) (2018) 118--130.

\bibitem{macmahon1912}
P.A.~MacMahon,
Memoir on the theory of partitions of numbers---Part~VI,
\emph{Phil.\ Trans.\ Royal Society~A} 211 (1912) 345--373.

\bibitem{macmahon1916}
P.A.~MacMahon,
\emph{Combinatory Analysis}, Vols.~1--2,
Cambridge University Press, 1915--1916.

\bibitem{marvel2011}
S.A.~Marvel, J.~Kleinberg, R.D.~Kleinberg, S.H.~Strogatz,
Continuous-time model of structural balance,
\emph{PNAS} 108\,(5) (2011) 1771--1776.

\bibitem{mckay2014}
B.D.~McKay, A.~Piperno,
Practical graph isomorphism, II,
\emph{Journal of Symbolic Computation} 60 (2014) 94--112.

\bibitem{stanley1999}
R.P.~Stanley,
\emph{Enumerative Combinatorics}, Volume~2,
Cambridge Studies in Advanced Mathematics No.~62,
Cambridge University Press, 1999.

\bibitem{vanditmarsch2017}
H.~van Ditmarsch, J.~van Eijck, P.~Pardo, R.~Ramezanian,
F.~Schwarzentruber,
Epistemic protocols for dynamic gossip,
\emph{Journal of Applied Logic} 20 (2017) 1--31.

\end{thebibliography}
\end{document}